\documentclass[10pt,b5paper]{article}
\usepackage[T1]{fontenc}
\usepackage[utf8]{inputenc}
\usepackage{hyperref}
\usepackage{amsmath,amssymb,amsthm}
\usepackage{geometry}
\usepackage{url}
\newcommand{\E}{\mathbb E}
\newcommand{\Bin}{\operatorname{Bin}}
\newcommand{\Poi}{\operatorname{Poi}}
\newcommand{\NB}{\operatorname{NB}}
\newcommand{\Hyp}{\operatorname{Hyp}}
\newcommand{\Z}{\mathbb Z}
\newcommand{\R}{\mathbb R}
\newcommand{\Bh}{\widehat B}

\theoremstyle{plain}
\newtheorem{theorem}{Theorem}[section]

\newtheorem{proposition}[theorem]{Proposition}
\newtheorem{corollary}[theorem]{Corollary}
\theoremstyle{remark}
\newtheorem{remark}[theorem]{Remark}
\theoremstyle{definition}

\title{The classical discrete laws near the mode:\\
complete local expansions for the negative binomial\\
and hypergeometric distributions}
\author{
N.~Elezovi\'c\\
Department of Applied Mathematics,\\
Faculty of Electrical Engineering and Computing,\\
University of Zagreb, 10000 Zagreb, Croatia\\
\texttt{neven.elezovic@fer.hr}
}
\date{\today}

\begin{document}
\maketitle

\begin{abstract}
	We derive complete local asymptotic expansions near the mode for the negative binomial and
	hypergeometric laws, complementing the binomial expansion obtained in the companion papers.
	The coefficients are given in closed Bernoulli-polynomial form and retain the exact lattice
	displacement from the mean.  For each law an entropy normalisation replaces the Bernoulli
	polynomials by reciprocal Bernoulli polynomials and makes the natural reflection symmetry
	visible.  The first-order coefficient also recovers the exact mode rule: in the Katz cases
	the vertex gives the classical threshold exactly, while in the hypergeometric case its
	explicit discrepancy is always too small to change the selected integer mode.
\end{abstract}

\noindent\textbf{2020 Mathematics Subject Classification.}
	41A60, 60C05, 11B68, 33B15.

\medskip\noindent\textbf{Keywords.}
	Mode; local limit theorem; negative binomial distribution; hypergeometric distribution;
	Bernoulli and reciprocal Bernoulli polynomials; Stirling series.

\section{Introduction}\label{sec:intro}

	Let $X$ be one of the classical discrete random variables.  Near its mode the mass
	$P(k)=\Pr\{X=k\}$ is of order $(2\pi\sigma^{2})^{-1/2}$, and the question of this paper is
	the \emph{complete} expansion in this regime: all orders, with closed coefficients, and with
	the arithmetic of the lattice --- the mean is in general not a lattice point --- carried
	exactly.

	For the binomial this was done in \cite{elezovic_mad,elezovic_mode}: writing
	$m=\lceil Np\rceil+r$ with $r$ a fixed integer, the mass has an expansion in integer powers
	of $N^{-1}$ whose coefficients are Bernoulli polynomials of the oscillating fractional
	displacement, organised by three normalisations --- naive, size-biased (pure Appell), and
	even (entropy) --- the last governed by the reciprocal Bernoulli polynomials
	$\Bh_n(t)=B_n(t)+\tfrac n2t^{n-1}$ of Kellner~\cite{kellner}, generated by
	$\tfrac z2\coth\tfrac z2$, which restore the $p\leftrightarrow q$ symmetry; and the exact
	mode rule $\lfloor(N+1)p\rfloor$ is already visible in the coefficient of $N^{-1}$.  The
	mean-absolute-deviation applications of these local expansions, for the whole family, are
	developed in \cite{elezovic_madfam}.

	Here we carry the local programme itself to the rest of the family.  The first part gives
	the complete expansions for the negative binomial and hypergeometric laws
	(Theorems~\ref{thm:nb-main} and~\ref{thm:hyp-main}).
	For $X\sim\NB(r,p)$ with $r\to\infty$ and for $X\sim\Hyp(N,K,n)$ with $N\to\infty$ and
	margins in fixed proportions, the mass at $k=\mu+t$, $t$ bounded, expands in integer powers
	of the large parameter with closed Bernoulli-polynomial coefficients.  The negative binomial
	needs three gamma factors with unequal scalings $\tfrac1p,\,1,\,\tfrac qp$; the
	hypergeometric needs nine, organised along the $2\times2$ table with margins
	$K,N-K,n,N-n$, whose independence identity ($ad=bc$ for the cell proportions) assembles all
	elementary constants.

	The second part concerns the even normalisations
	(Theorems~\ref{thm:nb-even} and~\ref{thm:hyp-even}).
	Normalising instead by the local (entropy) prefactor --- for the negative binomial
	$\bigl(2\pi\,k(k+r)/r\bigr)^{-1/2}$, for the hypergeometric the all-Stirling prefactor of
	the four cells --- replaces every $B_{n+1}$ of the displacement by $\Bh_{n+1}$,
	\emph{verbatim}.  This extends the binomial phenomenon of \cite{elezovic_mode} to the
	family.

	The third point is the reflection structure.  The binomial's
	$p\leftrightarrow q$ symmetry, made manifest by $\Bh$, generalises to the negative binomial
	as the \emph{formal negative reflection} $q\mapsto1/q$ (equivalently $p\mapsto-p/q$): the
	even coefficients are proved covariant under it combined with $t\mapsto-t$.  For the
	hypergeometric the reflection is the table transposition $k\mapsto n-k$,
	$K\leftrightarrow N-K$.

	Finally, the first-order coefficient determines the mode.  Let $t_c$ be the vertex of the
	quadratic part of the first-order coefficient.  For the three
	Katz laws (linear probability ratio) the vertex reproduces the classical mode threshold
	\emph{exactly}:
	\[
		\mu+t_c+\tfrac12=\lambda,\qquad (N+1)p,\qquad \frac{(r-1)q}{p}
	\]
	for the Poisson, binomial and negative binomial, so the mode is the unique integer in the
	unit window centred at $\mu+t_c$.  For the hypergeometric (quadratic Ord ratio) the vertex
	misses the classical threshold $(n+1)(K+1)/(N+2)$ by the explicit amount
	$-(2K-N)(2n-N)/\bigl(N^{2}(N+2)\bigr)$ --- nonzero unless a margin is exactly half the
	population --- but we prove this is always smaller than the distance of the threshold to
	the nearest integer, so the nearest-integer-to-vertex rule still yields the exact mode,
	with ties precisely at the two-mode configurations.

	The Poisson sits inside all of this as the degenerate centre of the family: one gamma
	factor, weight one, and (Proposition~\ref{prop:nb-seam}) the negative binomial expansion
	degenerates to it coefficientwise along the seam $rq\to\lambda$.

	The underlying frame is the Katz--Ord family: laws whose probability ratio
	$P(k+1)/P(k)$ is a rational function of $k$ of degree one (Katz: binomial, Poisson,
	negative binomial) or two (Ord: hypergeometric); see \cite[Ch.~2]{jkk},
	Katz~\cite{katz1965} and Ord~\cite{ord1967}.  Every structural statement below --- the size-bias identities, the
	degree of the mode-window analysis, the exactness or near-exactness of the vertex rule ---
	tracks this degree.

\section{Preliminaries}\label{sec:prelim}

	$B_n$ and $B_n(t)$ are the Bernoulli numbers and polynomials, $B_1=-\tfrac12$, with
	\begin{equation}\label{eq:appell}
	\begin{gathered}
		B_n(t+1)-B_n(t)=n\,t^{n-1},
		\qquad B_n(1-t)=(-1)^{n}B_n(t),\\
		B_n(1)=B_n\qquad(n\ne1).
	\end{gathered}
	\end{equation}
	The \emph{reciprocal Bernoulli polynomials} \cite{kellner} are
	\begin{equation}\label{eq:Bhat}
	\begin{gathered}
		\Bh_n(t):=\tfrac12\bigl[B_n(t)+B_n(t+1)\bigr]
			=B_n(t)+\tfrac n2\,t^{n-1}\quad(n\ge1),\\
		\sum_{n\ge0}\Bh_n(t)\frac{z^{n}}{n!}
			=\frac z2\coth\frac z2\,e^{tz},
	\end{gathered}
	\end{equation}
	an \emph{even} Appell sequence: $\Bh_n(-t)=(-1)^{n}\Bh_n(t)$ (DLMF~\cite[24.4.5]{dlmf}; see
	also \cite[\S3]{elezovic_mode}).

	Asymptotic statements are complete and uniform in the sense of \cite[\S1]{elezovic_mad}:
	for every fixed truncation order $M$ the remainder is $O(x^{-M-1})$ uniformly over the
	stated parameter compacts, and exponentials of truncated series are re-expanded through the
	standard Bell recursion.  Each theorem below expands $\log$ of the normalised mass; the mass
	itself follows by exponentiation and has a complete expansion of the same integer-power,
	Bernoulli-coefficient shape.  We use throughout the shifted Stirling series
	\begin{equation}\label{eq:stirling-shift}
		\log\Gamma(x+t)
		\sim\Bigl(x+t-\tfrac12\Bigr)\log x-x+\tfrac12\log2\pi
		+\sum_{n\ge1}
			\frac{(-1)^{n+1}B_{n+1}(t)}{n(n+1)}\,\frac1{x^{n}},
	\end{equation}
	uniformly for $t$ in compacta.  This is classical --- it is \cite[(5.11.8)]{dlmf}, and the
	corresponding expansion for a ratio of two gamma functions goes back to Tricomi and
	Erd\'elyi \cite{tricomi_erdelyi}; the Bernoulli polynomials of the shift are already present
	there.  What the present paper adds is not this identity but its complete assembly for the
	negative binomial and the hypergeometric, together with the mode-vertex consequences.  We
	shall use it through the expansion of a balanced gamma quotient with unequal
	scalings, \cite[Lemma 2.1]{elezovic_mad}, which we state in the following form: if
	$\sum_j\lambda_j=\sum_k\mu_k$, then
	\begin{multline}\label{eq:lemma}
		\log\frac{\prod_j\Gamma(\lambda_jx+u_j)}{\prod_k\Gamma(\mu_kx+v_k)}
		\sim \Theta x+U\log x+C\\
		\quad+\sum_{n\ge1}\frac{(-1)^{n+1}}{n(n+1)}
		\Bigl[\sum_j\frac{B_{n+1}(u_j)}{\lambda_j^{\,n}}
		-\sum_k\frac{B_{n+1}(v_k)}{\mu_k^{\,n}}\Bigr]\frac1{x^{n}},
	\end{multline}
	with the elementary constants $\Theta,U,C$ as given there and, for every $M$, a remainder
	$O(x^{-M-1})$ uniform for the shifts in compacta and the scalings in compacta of
	$(0,\infty)$.

\section{Size-biasing across the family}\label{sec:sizebias}

	The transfer of the binomial method rests on two ingredients: the size-bias identity, which
	controls the elementary (non-Bernoulli) content, and the gamma-quotient form of the pmf,
	which feeds \eqref{eq:lemma}.  We record the first for the family.

\begin{proposition}[Size biasing across the family]\label{prop:sizebias}
	Let $P_\theta(k)$ denote one of the four laws in this paper, with mean $\mu(\theta)$.  Then
	\begin{equation}\label{eq:sb-family}
		k\,P_\theta(k)=\mu(\theta)\,P_{\theta^\downarrow}(k-1),
	\end{equation}
	where the size-biased parameter $\theta^\downarrow$ is given by
	\[
	\begin{array}{c|c|c}
		\text{law} & \mu(\theta) & \theta^\downarrow\\ \hline
		\Bin(N,p) & Np & (N-1,p)\\[2pt]
		\Poi(\lambda) & \lambda & \lambda\\[2pt]
		\NB(r,p) & rq/p & (r+1,p)\\[2pt]
		\Hyp(N,K,n) & nK/N & (N-1,K-1,n-1).
	\end{array}
	\]
	The negative binomial is in the failures parametrisation,
	\[
		P(k;r,p)=\binom{k+r-1}{k}p^rq^k .
	\]
\end{proposition}

\begin{proof}
	For the binomial, \eqref{eq:sb-family} is the absorption
	$k\binom Nk=N\binom{N-1}{k-1}$; for the Poisson it is immediate.  For the negative binomial,
	\[
		k\binom{k+r-1}{k}
		=\frac{\Gamma(k+r)}{\Gamma(k)\,\Gamma(r)}
		=r\binom{k+r-1}{k-1},
	\]
	and
	$\binom{k+r-1}{k-1}=\binom{(k-1)+(r+1)-1}{k-1}$, so
	\[
		kP(k;r,p)=r\binom{k+r-1}{k-1}p^rq^k
		=\frac{rq}p\binom{(k-1)+r}{k-1}p^{r+1}q^{k-1}.
	\]
	For the hypergeometric, use $k\binom Kk=K\binom{K-1}{k-1}$,
	$\binom{N-K}{n-k}=\binom{(N-1)-(K-1)}{(n-1)-(k-1)}$ and
	$\binom Nn=\frac Nn\binom{N-1}{n-1}$.
\end{proof}

\begin{remark}[Provenance, and the recurrence behind it]\label{rem:katz-provenance}
	For the Katz subfamily, Proposition~\ref{prop:sizebias} is the $r=0$ case of the classical
	factorial-moment recurrence $(x+1)^{[r+1]}p_{x+1}=(\alpha+\beta x)(x+1)^{[r]}p_x$ that
	defines the family (see \cite[eq.\ (2.41)]{jkk}); ratio-based diagnostics on the same
	recurrence go back to Guldberg and Ottestad, and its systematic use is
	Katz~\cite{katz1965}.  The probabilistic reading --- ``the size-biased law is a unit
	parameter shift of the same family'' --- is standard.  What is used here is its
	\emph{asymptotic} role, as in \cite{elezovic_mode}: the identity is exactly the elementary
	tail of the local expansions below.
\end{remark}

\begin{remark}[One shift for the mode, two for the mean deviation]\label{rem:double-shift}
	For the hypergeometric, \eqref{eq:sb-family} shifts all three parameters down by one.  In the
	mean-absolute-deviation application \cite{elezovic_madfam} a \emph{second}, complementary
	shift appears --- $\E|X-\mu|=2\sigma^{2}P\{\Hyp(N-2,K-1,n-1)=\nu-1\}$, $\nu=\lceil\mu\rceil$ --- because the
	quadratic Ord ratio supplies two growing entries (the diagonal cells of the $2\times2$
	table), each with its own size-bias factor.  The same doubling will be visible below in the
	local expansion, where the diagonal cells carry the shifts $t+1$.
\end{remark}

	The Poisson local expansion, for later reference, is one application of
	\eqref{eq:stirling-shift}: with $k=\lambda+t$, $t$ in a compact set, $\Gamma(k+1)$ is the
	single large gamma factor and
	\begin{equation}\label{eq:poi-local}
		\log\Bigl[\sqrt{2\pi\lambda}\;P(k;\lambda)\Bigr]
		\sim-\sum_{n\ge1}\frac{(-1)^{n+1}B_{n+1}(t+1)}{n(n+1)}\,\frac1{\lambda^{n}} ,
	\end{equation}
	uniformly; the mean $\lambda$ is the natural centre, and the mode is
	$\lfloor\lambda\rfloor$.

\section{The negative binomial}\label{sec:nb}

	Let $X\sim\NB(r,p)$ in the failures parametrisation, $r>0$ real, $p$ in a compact subset of
	$(0,1)$, $q=1-p$,
	\[
		P(k)=\frac{\Gamma(k+r)}{\Gamma(r)\,\Gamma(k+1)}\,p^{r}q^{k},
		\qquad
		\mu=\frac{rq}{p},\qquad\sigma^{2}=\frac{rq}{p^{2}},
	\]
	and let $r\to\infty$.  Throughout, $k=\mu+t$ with $t$ in a fixed compact set $H\subset\R$
	(and $k\in\Z_{\ge0}$, which for large $r$ is compatible with any $H$).

	The law has been approached locally twice before, in a form worth stating precisely, since
	the difference from what follows is one of kind rather than of degree.  Govindarajulu
	\cite[Result 4.3.3]{govindarajulu} expands the mass uniformly in $k$ to $o(N^{-2})$ as the
	standard normal density and its derivatives $\varphi^{(3)},\dots,\varphi^{(7)}$ in the
	\emph{standardised} variable, with coefficients assembled from the cumulants of the
	geometric summand; his binomial, Poisson and hypergeometric sections
	\cite[\S\S4.1,4.2,4.4]{govindarajulu} are of the same Edgeworth type, the hypergeometric
	being reached through a binomial approximation.  Ouimet \cite{ouimet_nb} works in exactly
	the present regime and gives two explicit orders, $O((rp)^{-1/2})$ and $O((rp)^{-1})$, with
	a rigorous remainder, applying them to a refined continuity correction and to the median.
	Both are fixed-order expansions in a continuous standardised variable; neither carries
	Bernoulli polynomials, and neither addresses the mode.  Theorem~\ref{thm:nb-main} instead
	fixes the lattice displacement, gives \emph{all} orders in closed Bernoulli form, and
	thereby makes the mode-vertex statement of Theorem~\ref{thm:nb-mode} available at all.

\subsection{The complete expansion}

\begin{theorem}\label{thm:nb-main}
	As $r\to\infty$, uniformly for $p$ in compact subsets of $(0,1)$ and $t\in H$ with
	$\mu+t\in\Z_{\ge0}$,
	\begin{equation}\label{eq:nb-main}
	\begin{gathered}
		\log\Bigl[\sqrt{2\pi\sigma^{2}}\;P(\mu+t)\Bigr]
		\sim\sum_{n\ge1}\frac{A_n(t;p)}{r^{n}},\\
		A_n=\frac{(-1)^{n+1}}{n(n+1)}
		\Bigl[p^{n}B_{n+1}(t)
		-\Bigl(\frac pq\Bigr)^{n}B_{n+1}(t+1)-B_{n+1}\Bigr].
	\end{gathered}
	\end{equation}
	In particular
	\begin{equation}\label{eq:nb-a12}
		A_1(t;p)=-\frac{p^{2}}{2q}\,B_2(t)-\frac pq\,t-\frac1{12},
		\qquad
		A_2(t;p)=-\frac16\Bigl[p^{2}B_3(t)-\frac{p^{2}}{q^{2}}\,B_3(t+1)\Bigr],
	\end{equation}
\end{theorem}

	The $t^{2}$-part of $A_1/r$ is $-t^{2}/(2\sigma^{2})$; thus the Gaussian factor is already
	embedded in the logarithmic expansion.

\begin{proof}
	Since $k+r=\tfrac rp+t$ and $k+1=\tfrac{rq}p+(t+1)$,
	\[
		P(\mu+t)=\frac{\Gamma\bigl(\tfrac1p\,r+t\bigr)}
		{\Gamma(1\cdot r+0)\,\Gamma\bigl(\tfrac qp\,r+t+1\bigr)}\;p^{r}q^{\mu+t},
	\]
	a balanced quotient ($\tfrac1p-1-\tfrac qp=0$) in the format of \eqref{eq:lemma} with
	$x=r$, scalings $\bigl(\tfrac1p;\,1,\tfrac qp\bigr)$ and shifts $(t;\,0,\,t+1)$, all in
	compacta.  The series part of \eqref{eq:lemma} is the bracket of \eqref{eq:nb-main} with
	$B_{n+1}(0)=B_{n+1}$.  The elementary constants assemble exactly as in
	\cite[\S5]{elezovic_madfam}: the exponential rate cancels against $p^{r}q^{\mu}$,
	\[
		\Theta r+r\log p+\mu\log q
		=r\log p\Bigl(1-\frac1p+\frac qp\Bigr)=0
	\]
	after collecting the coefficients of $\log p,\log q$, leaving $q^{t}$.  Also $U=-\tfrac12$,
	and the constant $C$ of \eqref{eq:lemma} carries a compensating $q^{-t}$ (its $q$-dependent
	part is $-(t+\tfrac12)\log q$), so that
	\[
		e^{C}q^{t}\,r^{-1/2}
		=p(2\pi rq)^{-1/2}=(2\pi\sigma^{2})^{-1/2}.
	\]
	This gives \eqref{eq:nb-main}.  For \eqref{eq:nb-a12}, use
	$B_2(t+1)=B_2(t)+2t$, $p-\tfrac pq=-\tfrac{p^{2}}q$ and $\tfrac p{2q}\cdot2t=\tfrac pq t$:
	\[
		A_1=\tfrac12\Bigl[\Bigl(p-\frac pq\Bigr)B_2(t)-\frac pq\,2t-B_2\Bigr]
		=-\frac{p^{2}}{2q}B_2(t)-\frac pq\,t-\frac1{12}.
	\]
	Finally the $t^{2}$-coefficient of $A_1/r$ is $-p^{2}/(2qr)=-1/(2\sigma^{2})$.
\end{proof}

\subsection{The even normalisation, and the negative reflection}

\begin{theorem}\label{thm:nb-even}
	Under the hypotheses of Theorem~\ref{thm:nb-main},
	\begin{equation}\label{eq:nb-even}
	\begin{gathered}
		\log\Bigl[\sqrt{2\pi\,\frac{k(k+r)}{r}}\;P(k)\Bigr]
		\sim\sum_{n\ge1}\frac{\widehat A_n(t;p)}{r^{n}},\\
		\widehat A_n=\frac{(-1)^{n+1}}{n(n+1)}
		\Bigl[-B_{n+1}-\Bigl(\Bigl(\frac pq\Bigr)^{n}-p^{n}\Bigr)\Bh_{n+1}(t)\Bigr],
	\end{gathered}
	\end{equation}
	with $k=\mu+t$.
\end{theorem}

	Thus, relative to the local entropy prefactor, every Bernoulli polynomial of the displacement
	is replaced by the reciprocal Bernoulli polynomial.

\begin{proof}
	Write the passage from $\sigma^{2}$ to $k(k+r)/r$ as
	\[
		\tfrac12\log\frac{k(k+r)}{r\sigma^{2}}
		=\tfrac12\log\Bigl(1+\frac{pt}{rq}\Bigr)+\tfrac12\log\Bigl(1+\frac{pt}{r}\Bigr)
		=\tfrac12\sum_{n\ge1}\frac{(-1)^{n-1}t^{n}}{n\,r^{n}}
		\Bigl[\Bigl(\frac pq\Bigr)^{n}+p^{n}\Bigr],
	\]
	using $k=\tfrac{rq}p(1+\tfrac{pt}{rq})$ and $k+r=\tfrac rp(1+\tfrac{pt}r)$; the two series
	converge for $r$ large, uniformly on the compacts, and may be truncated at any order with
	error $O(r^{-M-1})$, as in \cite[Lemma 2.4]{elezovic_madfam}.  Hence
	$\widehat A_n=A_n+\tfrac{(-1)^{n-1}t^{n}}{2n}\bigl[(p/q)^{n}+p^{n}\bigr]$.  On the other
	hand, by \eqref{eq:appell},
	\[
		A_n=\frac{(-1)^{n+1}}{n(n+1)}
		\Bigl[\Bigl(p^{n}-\Bigl(\frac pq\Bigr)^{n}\Bigr)B_{n+1}(t)-B_{n+1}\Bigr]
		+\frac{(-1)^{n}}{n}\Bigl(\frac pq\Bigr)^{n}t^{n},
	\]
	and, by \eqref{eq:Bhat}, replacing $B_{n+1}(t)$ by $\Bh_{n+1}(t)$ in the bracket changes it
	by
	\[
		\frac{(-1)^{n+1}}{n(n+1)}\Bigl(p^{n}-\Bigl(\frac pq\Bigr)^{n}\Bigr)\frac{(n+1)t^{n}}2
		=\frac{(-1)^{n+1}t^{n}}{2n}\Bigl(p^{n}-\Bigl(\frac pq\Bigr)^{n}\Bigr).
	\]
	It remains to check that
	\[
		\frac{(-1)^{n}}{n}\Bigl(\frac pq\Bigr)^{n}t^{n}
		+\frac{(-1)^{n-1}t^{n}}{2n}\Bigl[\Bigl(\frac pq\Bigr)^{n}+p^{n}\Bigr]
		=\frac{(-1)^{n+1}t^{n}}{2n}\Bigl(p^{n}-\Bigl(\frac pq\Bigr)^{n}\Bigr),
	\]
	which is immediate: both sides equal $\tfrac{(-1)^{n-1}t^{n}}{2n}[p^{n}-(p/q)^{n}]$.
\end{proof}

\begin{remark}[The negative reflection $q\mapsto1/q$]\label{rem:nb-reflection}
	The binomial's even coefficients are invariant under $(t,p)\mapsto(-t,q)$
	\cite{elezovic_mode}.  The negative binomial has no finite complement, but it has the
	classical \emph{negative reflection}: under the formal substitution
	\[
		q\longmapsto\frac1q\qquad\Bigl(\text{hence }p=1-q\longmapsto-\frac pq\Bigr),
	\]
	the weight $w_n:=(p/q)^{n}-p^{n}$ (as it appears in \eqref{eq:nb-even}) transforms as
	\[
		w_n\longmapsto(-p)^{n}-\Bigl(-\frac pq\Bigr)^{n}=(-1)^{n+1}w_n ,
	\]
	while $\Bh_{n+1}(-t)=(-1)^{n+1}\Bh_{n+1}(t)$.  Hence the even coefficients
	\eqref{eq:nb-even} are \emph{covariant}:
	$\widehat A_n(t;p)$ is invariant under the combined reflection
	$(t,q)\mapsto(-t,1/q)$.  This is the substitution behind
	$\binom{k+r-1}k=(-1)^{k}\binom{-r}k$ --- the negative binomial as a binomial with negative
	index --- and it is the exact analogue of the binomial $p\leftrightarrow q$ symmetry,
	visible only in the $\Bh$ normalisation.  (The substitution leaves the parameter domain,
	so the statement is an algebraic covariance of the coefficient functions, not a symmetry
	of laws.)
\end{remark}

\subsection{The mode, exactly}

\begin{theorem}\label{thm:nb-mode}
	Let $t_c$ be the vertex of the quadratic part of $A_1(\,\cdot\,;p)$.  Then
	\begin{equation}\label{eq:nb-vertex}
		t_c=-\frac{1+q}{2p}=-\frac qp-\frac12,
		\qquad\text{so}\qquad
		\mu+t_c+\tfrac12=\frac{(r-1)q}{p} ,
	\end{equation}
	and the unit window $\bigl(\mu+t_c-\tfrac12,\ \mu+t_c+\tfrac12\bigr]$ contains exactly one
	integer, $\lfloor(r-1)q/p\rfloor$.  This integer is the upper mode of $\NB(r,p)$; two modes
	occur exactly when the endpoint $(r-1)q/p$ is an integer.
\end{theorem}

\begin{proof}
	From \eqref{eq:nb-a12}, $A_1'(t)=-\tfrac{p^{2}}{2q}(2t-1)-\tfrac pq$, which vanishes at
	$2t_c-1=-\tfrac2p$, i.e.\ $t_c=\tfrac12-\tfrac1p=-\tfrac{1+q}{2p}$; and
	$\mu+t_c+\tfrac12=\tfrac{rq}p-\tfrac qp=\tfrac{(r-1)q}p$.  A half-open unit interval
	contains exactly one integer, here $\lfloor(r-1)q/p\rfloor$.  That this integer is the
	upper mode is the classical ratio computation: $P(k)\ge P(k-1)$ iff
	$q(k+r-1)\ge k$, iff $k\le(r-1)q/p+\ldots$ --- precisely, $\frac{P(k)}{P(k-1)}
	=\frac{q(k+r-1)}k\ge1\iff k(1-q)\le q(r-1)\iff k\le\frac{(r-1)q}p$ --- so the largest
	maximiser is $\lfloor(r-1)q/p\rfloor$, with a tie at the preceding integer exactly when
	$(r-1)q/p\in\Z$.
\end{proof}

\begin{corollary}[The value at the mode]\label{cor:nb-modevalue}
	Let $k^{*}=\lfloor(r-1)q/p\rfloor$ and $f=\{(r-1)q/p\}$, so that
	$t^{*}=k^{*}-\mu=-\tfrac qp-f$.  Then
	\begin{equation}\label{eq:nb-modevalue}
		\sqrt{2\pi\sigma^{2}}\;P(k^{*})
		=1+\frac{5}{12\,r}-\frac{B_2(f)}{2\sigma^{2}}+O(r^{-2}) :
	\end{equation}
	the fractional part of the mode threshold enters exactly through $B_2(f)$, and the constant
	$\tfrac5{12}$ is independent of $p$.
\end{corollary}

\begin{proof}
	Evaluate $A_1$ at $t^{*}$.  With $B_2(t^{*})=(\tfrac qp+f)^{2}+(\tfrac qp+f)+\tfrac16$ and
	$-\tfrac pq t^{*}=1+\tfrac{pf}q$,
	\[
		A_1(t^{*})=-\frac{p^{2}}{2q}B_2(t^{*})+1+\frac{pf}q-\frac1{12} .
	\]
	Expanding and collecting the $f$-terms as
	$\tfrac{p^{2}f}{2q}-\tfrac{p^{2}f^{2}}{2q}=-\tfrac{p^{2}}{2q}\bigl(B_2(f)-\tfrac16\bigr)$,
	the $p$-dependent constants cancel and $A_1(t^{*})=\tfrac5{12}-\tfrac{p^{2}}{2q}B_2(f)$;
	finally $\tfrac{p^{2}}{2qr}=\tfrac1{2\sigma^{2}}$, and $e^{A_1/r}=1+A_1/r+O(r^{-2})$.
\end{proof}

\subsection{The Poisson seam}

\begin{proposition}\label{prop:nb-seam}
	Fix $t$ and $n\ge1$, and let $r\to\infty$, $q\to0$ with $rq=\lambda$ held fixed.  Then the
	$n$-th term of \eqref{eq:nb-main} converges to the $n$-th term of the Poisson expansion
	\eqref{eq:poi-local}:
	\[
		\frac{A_n(t;p)}{r^{n}}\;\longrightarrow\;
		-\frac{(-1)^{n+1}B_{n+1}(t+1)}{n(n+1)}\,\frac1{\lambda^{n}} ,
	\]
	and $\sigma^{2}=rq/p^{2}\to\lambda$: the negative binomial expansion degenerates
	coefficientwise to the Poisson one.
\end{proposition}

\begin{proof}
	In $A_n/r^{n}$ the three terms carry the factors
	\[
		\frac{p^{n}}{r^{n}}=\Bigl(\frac{pq}{\lambda}\Bigr)^{n}\to0,
		\qquad
		\frac{(p/q)^{n}}{r^{n}}=\Bigl(\frac p\lambda\Bigr)^{n}\to\frac1{\lambda^{n}},
		\qquad
		\frac{B_{n+1}}{r^{n}}=B_{n+1}\Bigl(\frac q\lambda\Bigr)^{n}\to0 ,
	\]
	since $p\to1$, $q\to0$.  Only the middle term survives, with the stated limit.
\end{proof}

\section{The hypergeometric}\label{sec:hyp}

	The normal approximation to the hypergeometric point probabilities is classical --- Nicholson
	\cite{nicholson} for the first-order theory with bounds, Molenaar \cite{molenaar} for the
	systematic survey of elementary approximations --- but we are not aware of a complete
	expansion, in the sense used here, anywhere in that literature; the nine-gamma organisation
	along the $2\times2$ table below appears to be new.  For the closely related multivariate
	problems, Ouimet has obtained precise local limit theorems and Le~Cam-distance bounds
	\cite{ouimet_multinomial,ouimet_lecam}, again to fixed order.

	Let $X\sim\Hyp(N,K,n)$,
	\[
		P(k)=\frac{\binom Kk\binom{N-K}{n-k}}{\binom Nn},\qquad
		\mu=\frac{nK}N,\qquad L:=N-K-n,
	\]
	with $N\to\infty$ and the margin proportions $\kappa=K/N$, $\eta=n/N$ in a compact subset
	of $(0,1)^{2}$.  Write the cell proportions of the $2\times2$ table at independence,
	\[
	\begin{gathered}
		a=\kappa\eta,\quad b=\kappa(1-\eta),\quad
		c=\eta(1-\kappa),\quad d=(1-\kappa)(1-\eta),\\
		a+b+c+d=1,\qquad ad=bc ,
	\end{gathered}
	\]
	so that at $k=\mu$ the four table entries $k,\,K-k,\,n-k,\,L+k$ are exactly
	$Na,\,Nb,\,Nc,\,Nd$.  Throughout, $k=\mu+t$ with $t$ in a fixed compact $H$ and
	$k\in\Z$.  Define
	\begin{align}
		\Sigma_0^{2}&:=\Bigl[\frac1{Na}+\frac1{Nb}+\frac1{Nc}+\frac1{Nd}\Bigr]^{-1}
		=N\,ad\notag\\
		&=\frac{nK(N-K)(N-n)}{N^{3}}
		=\sigma^{2}\,\frac{N-1}{N},
		\label{eq:sigma0}
	\end{align}
	the second equality by $ad=bc$ and $a+b+c+d=1$: the harmonic combination of the four cell
	counts at independence is $Nad$, the variance without its finite-population factor.

\subsection{The complete expansion}

\begin{theorem}\label{thm:hyp-main}
	As $N\to\infty$, uniformly for $(\kappa,\eta)$ in compact subsets of $(0,1)^{2}$ and
	$t\in H$ with $\mu+t\in\Z$,
	\begin{equation}\label{eq:hyp-main}
		\log\Bigl[\sqrt{2\pi\Sigma_0^{2}}\;P(\mu+t)\Bigr]
		\sim\sum_{m\ge1}\frac{A_m(t)}{N^{m}},
	\end{equation}
	with
	\begin{multline}\label{eq:hyp-Am}
		A_m=\frac{(-1)^{m+1}}{m(m+1)}
		\Bigl[(W_m-1)\,B_{m+1}
		-\bigl(a^{-m}+d^{-m}\bigr)B_{m+1}(t+1)\\
		-\bigl(b^{-m}+c^{-m}\bigr)B_{m+1}(1-t)\Bigr],
	\end{multline}
	where $W_m=\kappa^{-m}+(1-\kappa)^{-m}+\eta^{-m}+(1-\eta)^{-m}$.
\end{theorem}

	The $t^{2}$-part of $A_1/N$ is $-t^{2}/(2\Sigma_0^{2})$.  In \eqref{eq:hyp-Am}, the diagonal
	cells carry the shift $t+1$ and the off-diagonal cells the reflected shift $1-t$.

\begin{proof}
	Write $P(k)$ through nine gamma factors,
	\[
		P(k)
		=\frac{\Gamma(K+1)\Gamma(N-K+1)\Gamma(n+1)\Gamma(N-n+1)}
			{\Gamma(k+1)\Gamma(K-k+1)\Gamma(n-k+1)\Gamma(L+k+1)\Gamma(N+1)}
	\]
	with, at $k=Na+t$, the (scaling, shift) data
	\[
	\begin{array}{ll}
		\text{numerator:} &
		(\kappa,1),(1-\kappa,1),(\eta,1),(1-\eta,1),\\[2pt]
		\text{denominator:} &
		(a,t+1),(b,1-t),(c,1-t),(d,t+1),(1,1),
	\end{array}
	\]
	e.g.\ $K-k+1=Nb+(1-t)$ and $L+k+1=Nd+(t+1)$.  Both sides of the balance sum to $2$, so
	\eqref{eq:lemma} applies with $x=N$.  The elementary constants assemble through the two
	identities of the independence table: the rate vanishes because
	$\sum_{\text{cells}}(\text{cell})\log(\text{cell})
	=\sum_{\text{margins}}(\text{margin})\log(\text{margin})$ --- collect the four logarithms
	using $a+b=\kappa$, $a+c=\eta$, $c+d=1-\kappa$, $b+d=1-\eta$ --- and the constant terms
	combine, using $\kappa(1-\kappa)\eta(1-\eta)=ad\cdot bc=(ad)^{2}$, to
	$U=-\tfrac12$ and $e^{C_{\mathrm{el}}}N^{-1/2}=(2\pi Nad)^{-1/2}$, exactly as in
	\cite[\S6]{elezovic_madfam} (there at $t\in[0,1)$; the computation is identical for
	$t\in H$).  The series part of \eqref{eq:lemma} is \eqref{eq:hyp-Am}, with
	$B_{m+1}(1)=B_{m+1}$ collected into $(W_m-1)B_{m+1}$.  For the $t^{2}$-part: the
	quadratic terms of $B_2(t+1)$ and $B_2(1-t)$ are both $t^{2}$, so the $t^{2}$-coefficient
	of $A_1$ is $-\tfrac12\sum_{\text{cells}}(\text{cell})^{-1}=-N/(2\Sigma_0^{2})$ by
	\eqref{eq:sigma0}.
\end{proof}

\subsection{The even normalisation}

\begin{theorem}\label{thm:hyp-even}
	Under the hypotheses of Theorem~\ref{thm:hyp-main}, with $k=\mu+t$,
	\begin{equation}\label{eq:hyp-even}
		\log\Bigl[\sqrt{2\pi\,\frac{k\,(K-k)\,(n-k)\,(L+k)\,N}{K(N-K)\,n\,(N-n)}}\;P(k)\Bigr]
		\sim\sum_{m\ge1}\frac{\widehat A_m(t)}{N^{m}},
	\end{equation}
	with
	\begin{equation}\label{eq:hyp-Am-even}
	\begin{gathered}
		\widehat A_m=\frac{(-1)^{m+1}}{m(m+1)}
		\Bigl[(W_m-1)\,B_{m+1}-\widehat\Xi_m\,\Bh_{m+1}(t)\Bigr],\\
		\widehat\Xi_m=(a^{-m}+d^{-m})
			+(-1)^{m+1}(b^{-m}+c^{-m}).
	\end{gathered}
	\end{equation}
\end{theorem}

	Relative to the all-Stirling entropy prefactor of the four cells, the displacement enters only
	through the reciprocal Bernoulli polynomials.  The coefficients $\widehat A_m$ are covariant
	under the table reflection $(t,\kappa)\mapsto(-t,1-\kappa)$, equivalently
	$k\mapsto n-k$, $K\leftrightarrow N-K$.

\begin{proof}
	The passage from the prefactor of \eqref{eq:hyp-main} to that of \eqref{eq:hyp-even} is
	\[
	\begin{aligned}
		\tfrac12\log\frac{k(K-k)(n-k)(L+k)N}{K(N-K)n(N-n)\,\Sigma_0^{2}}
		={}&\tfrac12\Bigl[\log\Bigl(1+\frac t{Na}\Bigr)
			+\log\Bigl(1-\frac t{Nb}\Bigr)\\
		&\quad+\log\Bigl(1-\frac t{Nc}\Bigr)
			+\log\Bigl(1+\frac t{Nd}\Bigr)\Bigr],
	\end{aligned}
	\]
	because $K(N-K)n(N-n)=N^{4}ad\cdot(bc/ad)=N^{4}ad$ and
	$k(K-k)(n-k)(L+k)=N^{4}abcd\prod(1\pm t/(N\cdot\text{cell}))$, with $abcd=(ad)^{2}$.
	Expanding the four logarithms (convergent, truncatable at any order uniformly), the
	$m$-th coefficient added to $A_m$ is
	\[
		\frac{t^{m}}{2m}\Bigl[(-1)^{m-1}\bigl(a^{-m}+d^{-m}\bigr)
		-\bigl(b^{-m}+c^{-m}\bigr)\Bigr].
	\]
	On the other hand, writing \eqref{eq:hyp-Am} on the single argument $t$ by
	\eqref{eq:appell} --- $B_{m+1}(t+1)=B_{m+1}(t)+(m+1)t^{m}$ and
	$B_{m+1}(1-t)=(-1)^{m+1}B_{m+1}(t)$ --- and then replacing $B_{m+1}(t)$ by
	$\Bh_{m+1}(t)=B_{m+1}(t)+\tfrac{m+1}2t^{m}$ inside the resulting weight
	$\widehat\Xi_m$, the change is
	$\frac{(-1)^{m}}{2m}\widehat\Xi_m\,t^{m}$, and the elementary residue left by the first
	step is $\frac{(-1)^{m}}{m}\bigl(a^{-m}+d^{-m}\bigr)t^{m}$.  The identity to check,
	\[
	\begin{aligned}
		\frac{(-1)^{m}}{m}(a^{-m}+d^{-m})t^{m}
		&+\frac{t^{m}}{2m}
			\Bigl[(-1)^{m-1}(a^{-m}+d^{-m})-(b^{-m}+c^{-m})\Bigr]\\
		&=\frac{(-1)^{m}}{2m}\widehat\Xi_m\,t^{m},
	\end{aligned}
	\]
	reduces, after dividing by $t^{m}/(2m)$, to
	$(-1)^{m}(a^{-m}+d^{-m})-(b^{-m}+c^{-m})
	=(-1)^{m}\bigl[(a^{-m}+d^{-m})+(-1)^{m+1}(b^{-m}+c^{-m})\bigr]$, which is immediate.
	Covariance under the table reflection: the map $\kappa\mapsto1-\kappa$ exchanges
	$a\leftrightarrow c$ and $b\leftrightarrow d$, hence fixes $W_m$ and maps
	$\widehat\Xi_m\mapsto(c^{-m}+b^{-m})+(-1)^{m+1}(d^{-m}+a^{-m})
	=(-1)^{m+1}\widehat\Xi_m$, while $\Bh_{m+1}(-t)=(-1)^{m+1}\Bh_{m+1}(t)$; the two signs
	cancel.
\end{proof}

\subsection{The mode: the vertex misses the threshold, and it does not matter}

	The classical mode is
	\begin{equation}\label{eq:hyp-mode-classical}
		k^{*}=\Bigl\lfloor\frac{(n+1)(K+1)}{N+2}\Bigr\rfloor,
	\end{equation}
	from the ratio $\frac{P(k)}{P(k-1)}=\frac{(K-k+1)(n-k+1)}{k(L+k)}\ge1
	\iff(K+1)(n+1)\ge k(N+2)$ (the quadratic terms cancel), with two modes exactly when
	$(n+1)(K+1)/(N+2)\in\Z$.

\begin{theorem}\label{thm:hyp-mode}
	Let $t_c$ be the vertex of the quadratic part of $A_1$ of \eqref{eq:hyp-Am}.  Then
	\begin{equation}\label{eq:hyp-vertex}
	\begin{gathered}
		t_c=-\frac{(N-2n)(N-2K)}{2N^{2}},\\
		\mu+t_c+\tfrac12-\frac{(n+1)(K+1)}{N+2}
		=-\frac{(2K-N)(2n-N)}{N^{2}\,(N+2)} .
	\end{gathered}
	\end{equation}
	and the discrepancy satisfies
	\[
		\Bigl|\mu+t_c+\tfrac12-\frac{(n+1)(K+1)}{N+2}\Bigr|<\frac1{N+2} .
	\]
	Consequently, whenever $(n+1)(K+1)/(N+2)\notin\Z$, the unit window
	$\bigl(\mu+t_c-\tfrac12,\,\mu+t_c+\tfrac12\bigr]$ contains exactly one integer and it is
	the mode \eqref{eq:hyp-mode-classical}; when $(n+1)(K+1)/(N+2)\in\Z$ the window boundary
	falls within $1/(N+2)$ of the tied pair, and the integer it selects is one of the two
	modes.
\end{theorem}

\begin{proof}
	From \eqref{eq:hyp-Am},
	\[
		A_1=-\tfrac12\bigl[(a^{-1}+d^{-1})B_2(t+1)
		+(b^{-1}+c^{-1})B_2(1-t)\bigr]+\text{const},
	\]
	and since $B_2(t+1)=t^{2}+t+\tfrac16$, $B_2(1-t)=t^{2}-t+\tfrac16$, its
	quadratic-plus-linear part is $-\tfrac12\bigl[(A+B)t^{2}+(A-B)t\bigr]$ with
	$A=a^{-1}+d^{-1}$, $B=b^{-1}+c^{-1}$.  The vertex is
	\[
		t_c=-\frac{A-B}{2(A+B)} .
	\]
	Using $ad=bc$: $A-B=\dfrac{a+d-b-c}{ad}$ and $A+B=\dfrac{1}{ad}$, so
	$t_c=-\tfrac12(a+d-b-c)$.  Now
	\[
	\begin{aligned}
		a+d-b-c
		&=\kappa\eta+(1-\kappa)(1-\eta)-\kappa(1-\eta)-\eta(1-\kappa)\\
		&=(2\kappa-1)(2\eta-1)=\frac{(2K-N)(2n-N)}{N^{2}},
	\end{aligned}
	\]
	giving the first formula of \eqref{eq:hyp-vertex}.  For the second, write $u=2K-N$,
	$v=2n-N$, so $K=\tfrac{N+u}2$, $n=\tfrac{N+v}2$; then
	\[
	\begin{aligned}
		\mu+t_c+\tfrac12
		&=\frac{nK}N-\frac{uv}{2N^{2}}+\tfrac12,\\
		\frac{(n+1)(K+1)}{N+2}
		&=\frac{(N+u+2)(N+v+2)}{4(N+2)},
	\end{aligned}
	\]
	and a direct computation --- expand both over the common denominator $N^{2}(N+2)$ ---
	gives
	\[
		\mu+t_c+\tfrac12-\frac{(n+1)(K+1)}{N+2}=-\frac{uv}{N^{2}(N+2)} .
	\]
	Since $0<K<N$ and $0<n<N$, we have $|u|<N$ and $|v|<N$, so $|uv|<N^{2}$ and the
	discrepancy is $<1/(N+2)$ in absolute value, strictly.

	Finally, the threshold $x:=(n+1)(K+1)/(N+2)$ is a rational number with denominator
	dividing $N+2$; if $x\notin\Z$, its distance to every integer is at least $1/(N+2)$.  The
	open interval of radius $1/(N+2)$ about $x$ therefore contains no integer, and it
	contains $\mu+t_c+\tfrac12$; hence $\lfloor\mu+t_c+\tfrac12\rfloor=\lfloor x\rfloor$, and
	the unique integer in the unit window ending at $\mu+t_c+\tfrac12$ is
	$\lfloor x\rfloor=k^{*}$.  If $x\in\Z$, the same estimate places $\mu+t_c+\tfrac12$
	within $1/(N+2)$ of $x$, so the selected integer is $x$ or $x-1$: the two tied modes.
\end{proof}

\begin{remark}[A Katz--Ord dichotomy]\label{rem:dichotomy}
	For the three Katz laws the vertex rule is exact \emph{on the nose}:
	$\mu+t_c+\tfrac12$ \emph{equals} the classical threshold ($\lambda$; $(N+1)p$;
	$(r-1)q/p$).  For the Ord law the vertex misses the threshold by the explicit amount
	$-(2K-N)(2n-N)/(N^{2}(N+2))$, which vanishes only for a half-population margin
	($K=\tfrac N2$ or $n=\tfrac N2$) --- but the miss is provably below the resolution
	$1/(N+2)$ of the threshold's arithmetic, so the rule survives.  The factor $(2K-N)(2n-N)$ is, up to sign, the numerator of the
		classical hypergeometric skewness $\propto(N-2K)(N-2n)$; the new content is its exact
		packaging as the discrepancy and the sub-resolution bound.  The linear probability
	ratio is exactly captured by the first coefficient; the quadratic ratio is captured up to
	a harmless defect.  This is, in miniature, the same linear/quadratic divide that separates
	the Katz collapse from the hypergeometric one in \cite{elezovic_madfam}.
\end{remark}

\begin{remark}[Reading the mode off an expansion]\label{rem:kms}
	The device of locating a mode as the nearest integer to a value read from an asymptotic
	expansion, with a controlled defect, is that of Kabluchko, Marynych and Sulzbach
	\cite{kabluchko_ms} for the Ewens distribution and the Stirling numbers, where the mode is the
	nearest integer to $\theta\log n-\theta\,\Gamma'(\theta)/\Gamma(\theta)-\tfrac12$ --- the same
	$-\tfrac12$ shift as here.  There, however, the nearest-integer rule holds only on a
	density-one set of $n$ and fails for infinitely many; the content of
	Theorems~\ref{thm:nb-mode} and~\ref{thm:hyp-mode} is that for the Katz--Ord family the defect
	is not merely $o(1)$ but provably below the lattice resolution $1/(N+2)$, so the rule is exact
	for \emph{all} parameters.
\end{remark}

\begin{remark}[Symmetric point]\label{rem:hyp-parity}
	At a symmetric configuration ($t=0$ with $\mu\in\Z$) the odd-index Bernoulli values kill
	the even-$m$ margin terms ($B_{m+1}=0$), and the even normalisation \eqref{eq:hyp-Am-even}
	collapses to even $\Bh$-values: the expansion at the central point runs, as in the
	binomial and negative binomial cases, effectively in every second order.  The extreme
	instance is $\kappa=\eta=\tfrac12$ with $N=2K=2n$, where $P$ is symmetric about $\mu$.
\end{remark}

\section{The family in one table}\label{sec:family}

	\begin{center}\small
	\begin{tabular}{llll}
	\hline
	law & ratio deg & scalings (num.; den.) & weight of $B_{m+1}$/$\Bh_{m+1}$\\
	\hline
	$\Poi(\lambda)$ & 1 & $-$;\ $1$ & $1$\\[3pt]
	$\Bin(N,p)$ & 1 & $1$;\ $p,q$ &
		$p^{-m}+(-1)^{m+1}q^{-m}$\\[3pt]
	$\NB(r,p)$ & 1 & $\tfrac1p$;\ $1,\tfrac qp$ &
		$(p/q)^{m}-p^{m}$\\[3pt]
	$\Hyp(N,K,n)$ & 2 &
		\begin{tabular}[t]{@{}l@{}}
			$\kappa,1{-}\kappa,\eta,1{-}\eta$;\\
			$a,b,c,d,1$
		\end{tabular} &
		$\begin{array}[t]{@{}l@{}}
			(a^{-m}{+}d^{-m})\\
			{}+(-1)^{m+1}(b^{-m}{+}c^{-m})
		\end{array}$\\
	\hline
	\end{tabular}
	\end{center}

	In the weight column $w$ is the coefficient of $B_{m+1}(t)$ (resp.\ $\Bh_{m+1}(t)$) in
	the bracket $[(\mathrm{const})\,B_{m+1}-w\,B_{m+1}(t)]$, the factor $\tfrac{(-1)^{m+1}}{m(m+1)}$
	kept outside; for the negative binomial $w=(p/q)^m-p^m$, the same weight as in the even-normalisation theorem
	and, for the binomial, in \cite{elezovic_mode}.

	\begin{center}\small
	\begin{tabular}{lll}
	\hline
	law & $t_c$ & $\mu+t_c+\tfrac12$\\
	\hline
	$\Poi(\lambda)$ & $-\tfrac12$ & $\lambda$\\[3pt]
	$\Bin(N,p)$ & $\tfrac{p-q}2$ & $(N+1)p$\\[3pt]
	$\NB(r,p)$ & $-\tfrac qp-\tfrac12$ & $\tfrac{(r-1)q}p$\\[3pt]
	$\Hyp(N,K,n)$ & $-\tfrac{(N-2n)(N-2K)}{2N^{2}}$ &
		$\begin{array}[t]{@{}l@{}}
			\dfrac{(n+1)(K+1)}{N+2}\\
			{}-\dfrac{uv}{N^{2}(N+2)}
		\end{array}$\\
	\hline
	\end{tabular}
	\end{center}

	\noindent
	($u=2K-N$, $v=2n-N$.)  Three structural constants of the family:
	(i) the balance $\sum(\text{numerator scalings})=\sum(\text{denominator scalings})$ holds
	in every case, so \eqref{eq:lemma} applies with no exponentially large or small factors
	surviving;
	(ii) entries of the law that grow with $k$ carry the shift $t+1$ and enter the weights
	with $+$, entries that shrink carry $1-t$ and enter with the parity sign --- the binomial
	grammar, spoken by each law over its own cells;
	(iii) the even normalisation replaces $B$ by $\Bh$ verbatim and makes the law's reflection
	manifest ($p\leftrightarrow q$; $q\mapsto1/q$; the table transposition).

\section{Concluding remarks}\label{sec:conclusion}

	The mean-absolute-deviation companion \cite{elezovic_madfam} uses the expansions of
	Theorems~\ref{thm:nb-main} and~\ref{thm:hyp-main} at the single displacement
	$t=h\in[0,1)$; the present paper supplies the full local theory --- arbitrary bounded
	displacements, the even normalisations, the reflections, and the mode.

	The modes themselves are classical (see \cite[Chs.~4--6]{jkk}; for the binomial mode and the
		mode--mean window, Kaas and Buhrman \cite{kaas_buhrman}), and so is the engine: the
	shifted Stirling series \eqref{eq:stirling-shift} is \cite[(5.11.8)]{dlmf} and, for ratios,
	Tricomi--Erd\'elyi \cite{tricomi_erdelyi}; low-order local expansions for these laws exist
	in the literature \cite{govindarajulu,nicholson,molenaar,ouimet_nb}.  What we claim is the
	finer layer: the \emph{complete} expansions with coefficients identified in closed Bernoulli
	form, the $\Bh$ normalisations with their reflections, and the exact vertex theorems --- in
	particular the hypergeometric discrepancy identity \eqref{eq:hyp-vertex}.  The distinction
	is consistent across the prior local-limit literature for these laws: those expansions are
	Edgeworth series in a continuous standardised variable, carried to fixed order, with
	coefficients built from cumulants \cite{govindarajulu,nicholson,ouimet_nb}, whereas the
	displacement here is held on the lattice and every order is evaluated in closed Bernoulli
	form.  It is precisely that closed form which makes the mode readable off the first
	coefficient; the \emph{standard continuous} Edgeworth expansion in the standardised variable does not
	exhibit the lattice arithmetic directly --- although lattice Edgeworth series with
	continuity-corrected (Sheppard) cumulants do encode it \cite{kolassa_mccullagh} --- and so
	does not by itself yield Theorems~\ref{thm:nb-mode} and~\ref{thm:hyp-mode}.

	Two natural continuations remain.  First, the multinomial: the four-cell structure of
	Section~\ref{sec:hyp} is the $2\times2$ instance of the general contingency table with
	fixed margins, and the machinery extends.  Second, the regime $t\to\infty$ with the
	crossover into the Edgeworth gauge, for which the natural uniform objects are the
	incomplete beta (binomial, negative binomial) and its hypergeometric analogue.


\end{document}